\documentclass{article}
\usepackage[utf8]{inputenc}
\usepackage[]{graphicx}
\usepackage{amsthm,amsmath,amssymb,tikz,tikz-cd,amsmath,mathrsfs,mathtools,multicol,dirtytalk,tabularx,xy,lipsum,url,enumitem,cmll}
\usepackage{float}
\usepackage{enumitem} 

\usepackage[total={6.25in, 8in}]{geometry}
\usepackage{fancyhdr}
\usepackage{changepage}

\numberwithin{equation}{section}
\theoremstyle{plain}
\newtheorem{theorem}{Theorem}[section]
\newtheorem{lemma}[theorem]{Lemma}
\newtheorem{proposition}[theorem]{Proposition}
\newtheorem{corollary}[theorem]{Corollary}

\theoremstyle{definition}
\newtheorem{definition}[theorem]{Definition}

\newtheorem{example}[theorem]{Example}
\newtheorem{notation}[theorem]{Notation}

\title{Topological duality for orthomodular lattices\footnote{This article is forthcoming in \emph{Mathematical Logic Quarterly} (2023).}\footnote{Joseph McDonald was supported as a research assistant by the SSHRC-IG \#435--2019--0331 of Katalin Bimb\'o. We thank the anonymous referee for their very helpful comments. We also thank the attendees of the BLAST 2022 conference (Chapman University, Department of Mathematics) at which a talk based on an earlier version of this paper was presented. }}
\author{Joseph McDonald\footnote{Institute: University of Alberta, Department of Philosophy, Edmonton, Canada, Email: jsmcdon1@ualberta.ca}\hspace{.1cm} and Katalin Bimb\'o\footnote{Institute: University of Alberta, Department of Philosophy, Edmonton, Canada, Email: bimbo@ualberta.ca}}
\date{}

\begin{document}
\maketitle

\begin{abstract}
A class of ordered relational topological spaces is described, which we call \emph{orthomodular spaces}. Our construction of these spaces involves adding a topology to the class of orthomodular frames introduced by Hartonas, along the lines of Bimb\'o's topologization of the class of orthoframes employed by Goldblatt in his representation of ortholattices. We then prove that the category of orthomodular lattices and homomorphisms is dually equivalent to the category of orthomodular spaces and certain continuous frame morphisms, which we call \emph{continuous weak p-morphisms}. It is well-known that orthomodular lattices provide an algebraic semantics for the quantum logic $\mathcal{Q}$. Hence, as an application of our duality, we develop a topological semantics for $\mathcal{Q}$ using orthomodular spaces and prove soundness and completeness.    
\par
\vspace{.15cm}
\par
\vspace{.15cm}

\end{abstract}

\maketitle                 

\fancyhead{dfg}
\fancyhead[L]{Topological duality for orthomodular lattices}
\fancyhead[C]{}
\fancyhead[R]{McDonald, J., Bimb\'o, K.}
\section{Introduction}
Goldblatt in \cite{goldblatt1} proved that every ortholattice $\mathfrak{A}$ is isomorphic to the algebra $\mathcal{CO}(\mathcal{X})^{\dagger}$ of clopen $\perp$-stable subsets of a Stone space $\mathcal{X}$ endowed with an orthogonality relation $\perp$ that is irreflexive and symmetric. Bimb\'o in \cite{bimbo} introduced the class of orthospaces by adding a topology to the class of orthoframes used by Goldblatt and extended his topological representation of ortholattices to a full duality between the category $\mathbb{OL}$ of ortholattices and homomorphisms and the category $\mathbb{OS}$ of orthospaces and continuous orthoframe morphisms.\footnote{We use terms such as \say{category} here and later; however, the second author would have preferred to avoid talk of categories altogether.} We refer to McDonald and Yamamoto in \cite{mcdonald} for a constructive analogue of this duality using spectral spaces. Orthomodular lattices form a well-known subclass of ortholattices. Prototypical examples of orthomodular lattices include the lattice $\mathbb{C}(\mathcal{H})$ of closed linear subspaces of an infinite-dimensional separable Hilbert space $\mathcal{H}$ over the complex numbers. Hence, orthomodular lattices provide an algebraic foundation for the quantum logic $\mathcal{Q}$ (refer to \cite{goldblatt2,hartonas} for more details on $\mathcal{Q}$). Goldblatt in \cite{goldblatt3} proved that the condition of orthomodularity in $\mathbb{C}(\mathcal{H})$ is not elementary in the sense that it cannot be characterized by any first-order property of the relation of orthogonality between vectors in $\mathcal{H}$, where $\mathcal{H}$ is a pre-Hilbert space. This result raised the natural question as to whether there exists an elementary subclass of orthoframes that can capture the condition of orthomodularity.\footnote{Recall that a class $K$ of structures is an \emph{elementary class} provided there exists a first-order theory $T$ of signature $\sigma$ such that $K$ consists of all $\sigma$-structures satisfying $T$.} This question was answered in the affirmative by Hartonas in \cite{hartonas} who introduced an elementary class of partially ordered orthoframes, which he calls \emph{orthomodular frames}, and demonstrated $\mathcal{Q}$ to be sound and complete with respect to this class of structures. 
 
 In this paper, we describe a class of ordered relational topological spaces, which we call \emph{orthomodular spaces}. Our construction of these spaces involves adding a topology to the class of orthomodular frames along the lines of the topologization in \cite{bimbo} of the class of orthoframes. We first prove that the spectrum $\mathscr{S}_0(\mathfrak{A})$ of an orthomodular lattice $\mathfrak{A}$ determined by the filter space $\mathfrak{F}(\mathfrak{A})$ of $\mathfrak{A}$, along with a distinguished subset $\mathfrak{P}(\mathfrak{A})\subseteq\mathfrak{F}(\mathfrak{A})$ of principal filters of $\mathfrak{A}$, gives rise to an orthomodular space. We then prove that the algebra $\mathscr{A}_0(\mathcal{X})$ determined by the collection $\mathcal{CO}(\mathcal{X})^{\dagger}$ of clopen $\perp$-stable subsets of an orthomodular space $\mathcal{X}$ gives rise to an orthomodular lattice. This leads to a topological representation theorem which shows that every orthomodular lattice $\mathfrak{A}$ is isomorphic to $\mathscr{A}_0(\mathscr{S}_0(\mathfrak{A}))$. We then prove an algebraic realization theorem which shows that every orthomodular space $\mathcal{X}$ is relationally homeomorphic to $\mathscr{S}_0(\mathscr{A}_0(\mathcal{X}))$. Then, by working with a special class of continuous frame morphisms to accompany the orthomodular spaces, which we call \emph{continuous weak} $p$-\emph{morphisms}, we prove that every orthomodular lattice homomorphism $\phi\colon\mathfrak{A}\to\mathfrak{A}'$ determines a continuous weak $p$-morphism $\mathscr{S}_1(\phi)$ from the dual space $\mathscr{S}_0(\mathfrak{A}')$ of $\mathfrak{A}'$ to the dual space $\mathscr{S}_0(\mathfrak{A})$ of $\mathfrak{A}$. Conversely, we also prove that every continuous weak $p$-morphism $\psi\colon\mathcal{X}\to\mathcal{X}'$ between orthomodular spaces determines an orthomodular lattice homomorphism $\mathscr{A}_1(\psi)$ from the dual lattice $\mathscr{A}_0(\mathcal{X}')$ of $\mathcal{X}'$ to the dual lattice $\mathscr{A}_0(\mathcal{X})$ of $\mathcal{X}$. The main theorem of this paper then shows that the induced contravariant functors $\mathscr{S}_F:=\langle\mathscr{S}_0,\mathscr{S}_1\rangle\colon\mathbb{OML}\to\mathbb{OMS}$ and $\mathscr{A}_F:=\langle\mathscr{A}_0,\mathscr{A}_1\rangle\colon\mathbb{OMS}\to\mathbb{OML}$ extend naturally to a dual equivalence of categories. Finally, by exploiting the known algebraic semantics that orthomodular lattices provide for $\mathcal{Q}$, as an application of our duality, we develop a topological semantics for $\mathcal{Q}$ using orthomodular spaces and prove soundness and completeness.

 Although this is the first full Stone-type duality for orthomodular lattices, it is not the first investigation of their topological representation theory.\footnote{In general, unlike a topological representation for a given class of algebras combined with an algebraic realization of their dual spaces which establishes a duality at the level of objects (i.e., algebras and spaces), a full duality extends to morphisms (i.e., homomorphisms and continuous functions) by way of contravariant functors which constitute a dual equivalence of categories. We refer to Clark and Davey in \cite{clark} for standard examples of full dualities which fall under the classification of a Stone-type duality.

 It is worth pointing out the duality for orthomodular lattices developed by Cannon and Doring in \cite{cannon}. Their duality is however a generalization of a Stone-type duality as they work with techniques from topos theory whereby one considers certain spectral presheaves as the duals arising from orthomodular lattices.} Iturrioz in \cite{iturrioz} proved that every orthomodular lattice $\mathfrak{A}$ is isomorphic to the orthomodular lattice of clopen $B\hspace{-.15cm}\perp$-stable subsets of a Stone space whose orthoframe reduct satisfies the additional requirement that $z\perp x$ iff $z\perp y$ implies $x=y$.\footnote{See \cite[p. 510]{iturrioz} for a description of the $\textit{B}\hspace{-.1cm}\perp$-subsets of an orthoframe.} Moreover, Iturrioz described the class of topological spaces that characterize (up to relational homeomorphism) the space of proper filters used in her representation. Iturrioz in \cite{iturrioz1} later gave a topological representation of orthomodular lattices by means of closure spaces which precisely returns Stone's representation of Boolean algebras when the orthomodular lattice in question is distributive (i.e., a Boolean algebra). Unlike the representation theory in \cite{iturrioz,iturrioz1}, we work with spaces whose frame reducts are the elementary class of structures defined in \cite{hartonas}. Moreover, we work with the space of all filters of an orthomodular lattice $\mathfrak{A}$. That is, we include the improper filter $\omega$ containing the $0$-element of $\mathfrak{A}$ and as a result, the canonical representation map sends $0$ in $\mathfrak{A}$ to $\{\omega\}$ in $\mathscr{A}_0(\mathscr{S}_0(\mathfrak{A}))$ instead of $\emptyset$. This paper makes the following three contributions: First, it offers an alternative topological representation of orthomodular lattices and algebraic realization of their duals, compared to \cite{iturrioz,iturrioz1}. Second, it extends the existing Stone-type topological representation theory of orthomodular lattices to a full duality. Finally, it provides the first topological semantics known for $\mathcal{Q}$.

\section{Orthomodular lattices and orthomodular frames}
In this section, we exposit some basic aspects of ortholattices and orthomodular lattices. We then describe the elementary class of orthomodular frames introduced by Hartonas. For a more detailed treatment of these classes of structures, consult \cite{kalmbach} and \cite{hartonas}, respectively.  
\subsection{Ortholattices and orthomodular lattices}
\begin{definition}
An \emph{ortholattice} is an algebra $\mathfrak{A}=\langle A;\cdot,+,-,0,1\rangle$ of similarity type $\langle 2,2,1,0,0\rangle$ satisfying the following equations: 
\begin{multicols}{3}
\begin{enumerate}
    \item $a\cdot (b\cdot c)=(a\cdot b)\cdot c$
    \item $a+( b+c)=(a+b)+c$
    \item $a\cdot b=b\cdot a$
    \item $a+ b=b+ a$
    \item $a\cdot(b+ a)=a$
    \item $a+(b\cdot a)=a$
    \item $-(a\cdot b)=-a+ -b$
    \item $-(a+ b)=-a\cdot -b$
    \item $-(-a\cdot -b)=a+ b$
    \item $-(-a+ -b)=a\cdot b$
     \item $1\cdot a=a$
     \item $0+ a=a$
    \item $a\cdot-a=0$
    \item $a+-a=1$
    \end{enumerate}
    \end{multicols}
\end{definition}

\begin{definition}
A \emph{Boolean algebra} is an ortholattice $\mathfrak{A}=\langle A;\cdot,+,-,0,1\rangle$ satisfying conditions 1--14 in Definition 2.1 along with the following distribution laws: 
\
    \begin{enumerate}
 \item $a\cdot(b+ c)=(a\cdot b)+(a\cdot c)$
    \item $a+ (b\cdot c)=(a+ b)\cdot(a+ c)$
\end{enumerate}

\end{definition}
Although some of the equations in the above definitions of Boolean algebras and ortholattices are redundant, as some equations are derivable from others, they emphasize the fact that ortholattices are not necessarily distributive. Indeed, we have the following. 
\begin{proposition}
An ortholattice $\mathfrak{A}$ is a distributive lattice if and only if $\mathfrak{A}$ is a Boolean algebra. 
\end{proposition}
\begin{proof}
For the left-to-right direction, we note that if $\mathfrak{A}$ is an ortholattice in the sense of Definition 2.1 satisfying the distribution laws (i.e., equations 1 and 2 of Definition 2.2), then $\mathfrak{A}$ is exactly a Boolean algebra in the sense of Definition 2.2. The right-to-left direction is also trivial.  
\end{proof}
\begin{proposition}
For any ortholattice $\mathfrak{A}$, the bottom element $0$ of $\mathfrak{A}$ is definable as the orthocomplement of the top element $1$ of $\mathfrak{A}$, and vice versa. 
\end{proposition}
\begin{proof}
The following calculations prove the claim:
\[0=a\cdot-a=-(-a+--a)=-(-a+ a)=-(a+-a)=-1\]
\[1=a+-a=-(-a\cdot--a)=-(-a\cdot a)=-(a\cdot -a)=-0\]

Note that the equalities $(-a+--a)=-(-a+ a)$ and $-(-a\cdot--a)=-(-a\cdot a)$ in the above calculations rely on the fact that $-$ is an involution, which is easily provable from Definition 2.1. The remaining equalities are justified directly by Definition 2.1.

    Hence, by equation 9 in Definition 2.1, ortholattices are presentable in the reduced signature $\langle A;\cdot,-,0\rangle$ in which joins can be defined as the De Morgan duals of meets. 
\end{proof}
 The above proposition allows us to simplify our description of ortholattices in the following manner. 
\begin{definition}
An \emph{ortholattice} is an algebra $\mathfrak{A}=\langle A;\cdot,-,0\rangle$ of similarity type $\langle 2,1,0\rangle$ such that:
\begin{enumerate}
    \item $\langle A;\cdot,0\rangle$ is a meet-semilattice with a least element $0$ 
    \item the operator $-\colon A\to A$ is an \emph{orthocomplementation}, i.e., 
    \begin{enumerate}
    \item $a\cdot -a=0$ 
  \item $a\leq b\Rightarrow -b\leq -a$
        \item $a=--a$
    \end{enumerate}
\end{enumerate}
\end{definition}
\begin{definition}
An \emph{orthomodular lattice} is an ortholattice $\mathfrak{A}=\langle A;\cdot,-,0\rangle$ satisfying the following quasi-inequation, known as the \emph{orthomodularity law}: 
 \[a\leq b\Rightarrow b=a+(-a\cdot b)\]
\end{definition}
\begin{example}
Let $\mathcal{H}$ be an infinite-dimensional separable complex Hilbert space and let $\mathbb{C}(\mathcal{H})$ be the collection of closed linear subspaces of $\mathcal{H}$. Then the algebra $\langle\mathbb{C}(\mathcal{H});\cap,\hspace{.025cm}^{\perp},\emptyset\rangle$ is an orthomodular lattice ordered by subspace inclusion. Notice that for each closed linear subspace $U\subseteq\mathcal{H}$, its orthogonal complement is defined by $U^{\perp}:=\{x\in\mathcal{H}:\langle x,y\rangle=0\hspace{.1cm}\text{for all}\hspace{.1cm}y\in U\}$, where $\langle-,-\rangle:\mathcal{H}\times\mathcal{H}\to\mathbb{C}$ is the inner product on $\mathcal{H}$. For a detailed treatment of Hilbert spaces, we refer to Cohen in \cite{cohen}.  
\end{example}

  Ortholattices, as well as the orthomodular lattices, are equationally definable; hence, by Birkhoff's Varieties Theorem, the classes of these algebras are varieties, that is, they are closed under the formation of homomorphic images, subalgebras, and direct products.

The following characterization of orthomodular lattices is well-known and will be exploited later on. 

\begin{proposition}\label{oml}
For any ortholattice $\mathfrak{A}$, the following conditions are equivalent: 
\begin{enumerate}
    \item $\mathfrak{A}$ is an orthomodular lattice;
    \item if $a\leq b$ and $-a\cdot b=0$, then $a=b$.
\end{enumerate}
\end{proposition}
\begin{proof}
See Beran \cite[Theorem 3.1, p. 40]{beran}. 
\end{proof}
\begin{example}
The Hasse diagrams in Figure \ref{hasse diagrams} depict 3 different kinds of ortholattices. The lattice $\text{O}_{2\times 2}$ on the left is a distributive ortholattice and hence by Proposition 2.3, it is a Boolean algebra. The lattice $\text{O}_6$ in the middle is a nondistributive ortholattice but is not an orthomodular lattice. The lattice $\text{OML}_6$ on the right is an orthomodular lattice.   
\end{example}

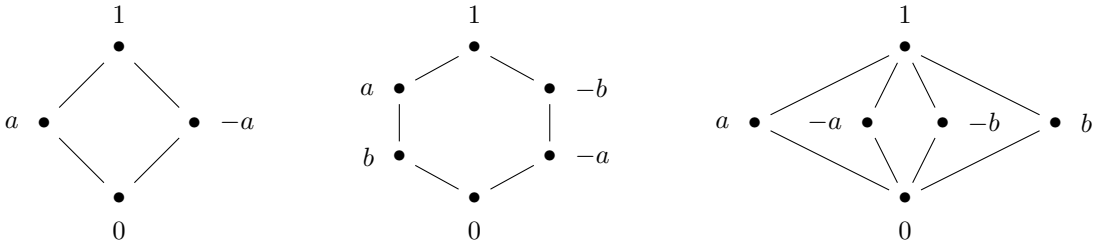
\begin{figure}[htbp]
\begin{tikzpicture}
  \node (a) at (0,1) [label=above:$1$] {$\bullet$};
  \node (b) at (1,0) [label=right:$-a$] {$\bullet$};
  \node (c) at (-1,0) [label=left:$a$] {$\bullet$};
  \node (d) at (0,-1) [label=below:$0$] {$\bullet$};
  \draw (a) -- (b) (b) -- (d) (d) -- (c) (c) -- (a) ;
  \draw[preaction={draw=white, -,line width=6pt}];
\end{tikzpicture}
\hskip 3em
  \begin{tikzpicture}
  \node (a) at (0,1) [label=above:$1$] {$\bullet$};
  \node (d) at (0,-1) [label=below:$0$] {$\bullet$};
  \node (e) at (-1,-.45) [label=left:$b$] {$\bullet$};
  \node (f) at (-1,.45) [label=left:$a$] {$\bullet$};
    \node (g) at (1,-.45) [label=right:$-a$] {$\bullet$};
  \node (h) at (1,.45) [label=right:$-b$] {$\bullet$};
  \draw (e) -- (f) (g) -- (h) (e) -- (d) (g) -- (d) (f) -- (a) (h) -- (a) (a);
  \draw[preaction={draw=white, -,line width=6pt}];
  \end{tikzpicture}
\hskip 3em 
\begin{tikzpicture}
  \node (a) at (0,1) [label=above:$1$] {$\bullet$};
  \node (b) at (2,0) [label=right:$b$] {$\bullet$};
  \node (c) at (-2,0) [label=left:$a$] {$\bullet$};
  \node (d) at (0,-1) [label=below:$0$] {$\bullet$};
  \node (i) at (.5,0) [label=right:$-b$] {$\bullet$};
  \node (j) at (-.5,0) [label=left:$-a$] {$\bullet$};
  \draw  (i) -- (a) (j) -- (a) (i) -- (d) (j) -- (d) (d) -- (b) (d) -- (c) (b) -- (a) (c) -- (a);
  \draw[preaction={draw=white, -,line width=6pt}];
\end{tikzpicture}
\caption{The lattices $\text{O}_{2\times 2}$, $\text{O}_6$, and $\text{OML}_6$}\label{hasse diagrams}
\end{figure}

\subsection{Orthomodular frames}
Orthomodular frames form a class of ordered relational structures of the following shape $\langle X;\preceq,\perp,\Omega\rangle$. The intended interpretation of the partial order reduct $\langle X;\preceq\rangle$ is that of an information ordering $\preceq$ over a set $X$ of information states. Information at a state $x\in X$ is to be viewed as information acquired as the result of an experiment or measurement performed on a quantum mechanical system. Two information states are to be viewed as orthogonal (which we denote by $x\perp y$) if their informational contents contradict each other.

\begin{definition}[Hartonas \cite{hartonas}]\label{omf}
An \emph{orthomodular frame} is an ordered relational structure $\langle X;\preceq,\perp,\Omega\rangle$ with relations of type $\preceq,{\perp}\subseteq X\times X$ and a distinguished subset $\Omega\subseteq X$ satisfying the following conditions: 
\begin{enumerate}
    \item $\langle X;\preceq\rangle$ is a meet-semilattice
    \item $\langle X;\preceq\rangle$ has an upper bound $\omega$ such that $\omega\in\Omega$ 
    \item for each point $x\in\Omega$, the set $\{y\in X: x\perp y\}$ is generated as the principal upset of a single point $z\in\Omega$. Moreover, the set $\{v\in X:z\perp v\}$ generated by $z$ is in the principal upset generated by $x$
     \item $\forall x\forall y((x\in\Omega\wedge y\in\Omega)\Rightarrow((y\preceq x\wedge\forall z(y\preceq z\wedge z\perp x)\Rightarrow z=\omega)\Rightarrow x=y))$
     \item $\forall x(x\perp x\Rightarrow x=\omega)$
     \item $\forall x\forall y((x\in\Omega\wedge y\in\Omega)\Rightarrow x\cap y\in\Omega)$ where $x\cap y:=\inf\{x,y\}$
     \item $\forall x(x\perp\omega)$
     \item $\forall x\forall y(x\perp y\Rightarrow y\perp x)$
     \item $\forall x\forall y\forall z((x\perp y\wedge x\preceq z)\Rightarrow z\perp y)$.
\end{enumerate}
\end{definition}
Notice that we notationally distinguish between the partial ordering $\preceq$ of an orthomodular frame and the partial ordering $\leq$ of an orthomodular lattice. 

Conditions 1 and 2 in the definition above guarantees that the information ordering $\langle X;\preceq\rangle$ forms a bounded meet-semilattice. In condition 3, \emph{the principal upset} of a state $x$ is the set $\{y\in X:x\preceq y\}$. For the remainder of this work, we denote the principal upset of a point $x$ by ${\uparrow}x$. Hence condition 3 can be alternatively formulated by the following first-order formula: 
\[\forall x(x\in\Omega\Rightarrow\exists z(z\in\Omega\wedge\forall y(x\perp y\Leftrightarrow z\preceq y)\wedge\forall u(z\perp u\Rightarrow x\preceq u)))\]

In \cite{hartonas}, condition 4 in the above definition is called the \emph{orthomodularity axiom} and gives a relational characterization of the orthomodularity condition stated within condition 2 of Proposition 2.8. Condition 5 asserts that $\omega$ is the only information state orthogonal to itself with condition 7 guaranteeing that $\omega$ is orthogonal to every state so that $\omega$ is the unique inconsistent state. Condition 6 is simply the requirement that the distinguished subset $\Omega$ be closed under finite infima, with condition 8 imposing the symmetry of $\perp$. Lastly, condition 9 is the requirement that any two orthogonal states are also orthogonal to each state contained within the other's principal upset. Notice that orthomodular frames form an elementary class of relational structures, defined in a first-order language over the signature $\{\preceq,\perp,\Omega\}$.
 
 \begin{definition}
 Let $\mathcal{X}=\langle X;\preceq,\perp,\Omega\rangle$ be an orthomodular frame and let $U\subseteq X$. Then, 
 \begin{enumerate}
     \item given $\wp(X):=\{U:U\subseteq X\}$, we define $^{\perp}\colon\wp(X)\to\wp(X)$ as follows: \[U^{\perp}:=\{x\in X:\forall y(y\in U\Rightarrow x\perp y)\}=\{x\in X:x\perp U\}\]
     \item $U$ is $\perp$-\emph{stable} if $U=U^{\perp\perp}$ where $U^{\perp\perp}:=(U^{\perp})^{\perp}$.
 \end{enumerate}
 We refer to $U^{\perp}$ as the \emph{orthogonal complement} of $U$. 
 \end{definition}
\section{Topological representation of orthomodular lattices}
\begin{notation}
 Let $\mathcal{X}=\langle X;\preceq,\perp,\Omega,\mathcal{T}_X\rangle$ be an ordered relational topological space such that $\langle X;\preceq,\perp,\Omega\rangle$ is an orthomodular frame and $\mathcal{T}_X$ is a topology on $X$. Then, we define the following subsets of $\mathcal{X}$ as follows:    
 \begin{enumerate}
    \item $\mathcal{C}(\mathcal{X}):=\{U\in\wp(X):\text{$U$ is closed in}\hspace{.1cm} \mathcal{X}\}$
     \item $\mathcal{O}(\mathcal{X}):=\{U\in\wp(X):\text{$U$ is open in}\hspace{.1cm}\mathcal{X}\}$
     \item $(\mathcal{X})^{\dagger}:=\{U\in\wp(X):\text{$U$ is}\hspace{.025cm}\perp\hspace{-.1cm}\text{-stable}\hspace{.1cm}\text{in}\hspace{.1cm}\mathcal{X}\}$
     \item $\mathcal{CO}(\mathcal{X}):=\mathcal{C}(\mathcal{X})\cap\mathcal{O}(\mathcal{X})$
     \item $\mathcal{CO}(\mathcal{X})^{\dagger}:=\mathcal{CO}(\mathcal{X})\cap(\mathcal{X})^{\dagger}$
 \end{enumerate}

\end{notation}
Building on the definition of orthospaces in \cite{bimbo}, we now describe the class of ordered relational topological spaces that will be shown to arise as the dual spaces of the orthomodular lattices. 
 \begin{definition}\label{oms}
 An \emph{orthomodular space} is an ordered relational topological space $\mathcal{X}=\langle X;\preceq,\perp,\Omega,\mathcal{T}_X\rangle$ satisfying the following conditions: 
 \begin{enumerate}
 \item $\langle X;\preceq,\perp,\Omega\rangle$ is an orthomodular frame;
 \item $\langle X;\mathcal{T}_X\rangle$ is a compact topological space; 
 \item $\mathcal{CO}(\mathcal{X})^{\dagger}$ is closed under $\cap$ and $^{\perp}$;
 \item each $U\in\mathcal{CO}(\mathcal{X})^{\dagger}$ is of the form ${\uparrow}z$ for some $z\in\Omega$;
 \item if $x\not\preceq y$, then there exists $U\in\mathcal{CO}(\mathcal{X})^{\dagger}$ such that $x\in U$ and $y\not\in U$;
 \item if $x\perp y$, then there exists $U\in\mathcal{CO}(\mathcal{X})^{\dagger}$ such that $x\in U$ and $y\in U^{\perp}$.
 \end{enumerate}
 \end{definition}

There are some key similarities between orthomodular spaces (as described above) and orthospaces (as described in \cite[Definition 3.2]{bimbo}) that are worth noting. First, both classes of spaces can be viewed as ordered relational Stone spaces whose topologies are generated by the set of clopen $\perp$-stable sets and their set-theoretic complements. Second, both classes of spaces satisfy a Priestley-style separation axiom with respect to clopen $\perp$-stable sets, i.e., condition 6 above and condition 1 of Definition 3.2 in \cite{bimbo} (see \cite{priestley} for more details on Priestley spaces). Lastly, by condition 6 of orthomodular spaces and condition 4 of orthospaces in \cite{bimbo}, both class of spaces allow one to topologically distinguish orthogonal pairs of points in the sense if $x$ is orthogonal to $y$, one can find some clopen $\perp$-stable set containing one such that its orthogonal complement contains the other. However, unlike orthospaces, orthomodular spaces include a distinguished subset $\Omega$ which is subject to conditions 2, 3, and 4 of Definition 2.10. Second, whereas the underlying frame reduct of an orthospace is an orthoframe, namely a set along with a binary relation that is irreflexive and symmetric, the underlying frame reduct of an orthomodular space is that of an orthomodular frame, which has a much more fine-grained relational structure. For instance, whereas the relation $\perp$ in orthospace is irreflexive on the entire carrier set, the irreflexivity of $\perp$ in orthomodular space is restricted to $X\setminus\{\omega\}$.

Recall that for a meet-semilattice $\mathfrak{A}=\langle A;\cdot,0\rangle$, a \emph{filter} is a non-empty subset $x\subseteq A$ that is upward closed and closed under finite meets. Moreover, recall that a filter $x\subseteq A$ is \emph{proper} if $x\not=A$ and \emph{improper} if $0\in x$. For the remainder of this work, by a filter, we mean a non-empty and possibly improper filter. Note that for any element $a\in A$, its principal upset ${\uparrow}a$ is a filter, known as the \emph{principal filter} generated by $a$. It is easy to see that in a finite lattice $\mathfrak{A}=\langle A;\cdot,+\rangle$, every filter $x\subseteq\mathfrak{A}$ is principal, i.e., $x={\uparrow}a$ for some $a\in\ A$.    

If $\mathcal{X}=\langle X;\mathcal{T}_X\rangle$ is a topological space and $\mathcal{B}_X$ is a basis for $\mathcal{X}$, then a family $\mathcal{C}:=\{U\}_{i\in I}\subseteq\mathcal{B}_X$ is a \emph{basic open cover} of a subset $V\subseteq X$ if $V\subseteq\bigcup_{i\in I}U_i$. Moreover, a finite subset of $\mathcal{C}$ whose union still contains $V$ is a \emph{finite subcover}. A topological space $\mathcal{X}$ is then said to be \emph{compact} provided every basic open cover has a finite subcover. Subbasic open covers are defined analogously but with respect to the subbasis of a topological space. Our compactness proof in Lemma 3.5 relies on the well-known characterization of compact spaces. 
\begin{theorem}[Alexander's Subbase Theorem]
Let $\mathcal{X}=\langle X;\mathcal{T}_X\rangle$ be a topological space and let $\mathcal{S}$ be a subbasis of $\mathcal{X}$. Then $\mathcal{X}$ is a compact topological space if and only if every basic open cover of $\mathcal{X}$ by elements of $\mathcal{S}$ has a finite subcover. 
\end{theorem}
\begin{proof}
The proof relies on the Axiom of Choice by way of the Boolean Prime Ideal Theorem. See for instance Davey and Priestley \cite[p. 278--279]{davey}.  
\end{proof}
\begin{definition}
If $\mathfrak{A}=\langle A;\cdot,-,0\rangle$ is an orthomodular lattice, then the ordered relational topological space $\mathscr{S}_0(\mathfrak{A})=\langle\mathfrak{F}(\mathfrak{A});\subseteq,\perp_A,\mathfrak{P}(\mathfrak{A}),\mathcal{T}(\mathcal{S})\rangle$ is defined as follows: 
\begin{enumerate}
    \item $\mathfrak{F}(\mathfrak{A})$ is the collection of all filters of $\mathfrak{A}$;
    \item $\mathfrak{P}(\mathfrak{A})$ is the collection of all principal filters of $\mathfrak{A}$;
    \item $\subseteq$ is subset inclusion;
    \item for all $x,y\in\mathfrak{F}(\mathfrak{A})$, define $x\perp_Ay$ iff there exists $a\in A$ such that $a\in x$ and $-a\in y$;
    \item $\mathcal{T}(\mathcal{S})$ is the topology on $\mathfrak{F}(\mathfrak{A})$ generated by the subbasis: \[\mathcal{S}=\{h(a): a\in A\}\cup\{\mathfrak{F}(\mathfrak{A})\setminus h(a): a\in A\}\] such that: \[\mathfrak{F}(\mathfrak{A})\setminus h(a)=\{x\in\mathfrak{F}(\mathfrak{A}):x\not\in h(a)\}\] and $h\colon \mathfrak{A}\to \wp(\mathfrak{F}(\mathfrak{A}))$ is a function defined by: \[h(a)=\{x\in\mathfrak{F}(\mathfrak{A}): a\in x\}.\]
\end{enumerate}
\end{definition}
The following lemma verifies that the topological space constructed in Definition 3.4 from the filter spectrum of an orthomodular lattice is an orthomodular space. 

\begin{lemma}[Dual space of an orthomodular lattice]\label{lattice to space}
If $\mathfrak{A}=\langle A;\cdot,-,0\rangle$ is an orthomodular lattice, then the ordered relational topological space $\mathscr{S}_0(\mathfrak{A})=\langle\mathfrak{F}(\mathfrak{A});\subseteq,\perp_A,\mathfrak{P}(\mathfrak{A}),\mathcal{T}(\mathcal{S})\rangle$ is an orthomodular space.
\end{lemma}
\begin{proof} 
The result that $\langle\mathfrak{F}(\mathfrak{A});\subseteq,\perp_A,\mathfrak{P}(\mathfrak{A})\rangle$ is an orthomodular frame is similar to the completeness proof for orthomodular quantum logic given by Hartonas in \cite{hartonas}. We briefly sketch the steps.

1.) It is well known that $\langle\mathfrak{F}(\mathfrak{A});\subseteq\rangle$ is a complete lattice and hence a meet-semilattice, so condition 1 of Definition \ref{omf} is satisfied.
   
2.) For condition 2, let $\omega$ be the improper filter in $\mathfrak{A}$ so that $\omega=A$. Therefore $x\subseteq \omega$ for each filter $x\in\mathfrak{F}(\mathfrak{A})$ and hence $\omega$ is an upper bound of $\langle\mathfrak{F}(\mathfrak{A});\subseteq\rangle$. Moreover $\omega\in\mathfrak{P}(\mathfrak{A})$ since $\omega={\uparrow}0$. Notice that $\omega$ is the least upper bound of $\langle\mathfrak{F}(\mathfrak{A});\subseteq\rangle$, and hence unique. Thus $\langle\mathfrak{F}(\mathfrak{A});\subseteq\rangle$ is in fact a bounded meet-semilattice.     

3. and 4.) Conditions 3 and 4 of Definition \ref{omf} are verified by using Hartonas' Lemmas 7 and 9 in \cite{hartonas}.

5.) For Condition 5, assume $x\perp_A x$. Then by the definition of $\perp_A$, there exists $a\in A$ such that $a,-a\in x$ and hence $a\cdot -a\in x$ since $x$ is a filter. Notice that $a\cdot -a=0$ and hence $0\in x$. Therefore, $x$ is the improper filter in $\mathfrak{A}$ and thus $x=\omega$.

6.) For condition 6, let $x,y\in\mathfrak{P}(\mathfrak{A})$ so that $x={\uparrow}a$ and $y={\uparrow}b$ for some $a,b\in A$. Then: 
\begin{align*}
    {\uparrow}a\cap{\uparrow}b&=\{c\in A:c\in{\uparrow}a\wedge c\in{\uparrow}b\}\tag{by the definition of $\cap$}
    \\&=\{c\in A: a\leq c\wedge b\leq c\}\tag{by the definition of $\uparrow$}
    \\&=\{c\in A:a\cdot b\leq c\}\tag{by the definition of $\cdot$}
    \\&=\{c\in A:c\in{\uparrow}(a\cdot b)\}\tag{by the definition of $\uparrow$}
\end{align*}
 Since ${\uparrow}(a\cdot b)\in\mathfrak{P}(\mathfrak{A})$, condition 6 is satisfied. 

7.) For condition 7, choose any filter $x\in\mathfrak{F}(\mathfrak{A})$. Clearly, for any element $a\in x$, we have $-a\in\omega$ since $\omega={\uparrow}0=A$ and $\mathfrak{A}$ is a complemented lattice. By the definition of $\perp_A$ we thus have $x\perp_A\omega$.

8.) For condition 8, assume $x\perp_Ay$. Then there exists $a\in A$ such that $a\in x$ and $-a\in y$. However $a=--a$, since $-$ is an involution, and hence, $--a\in x$. By the definition of $\perp_A$, we have $y\perp_Ax$. 

9.) For Condition 9, let $x,y,z\in\mathfrak{F}(\mathfrak{A})$ and assume $x\perp_Ay$ and $x\subseteq z$. By the former hypothesis, there exists some $a\in A$ such that $a\in x$ and $-a\in y$. By the latter hypothesis, we then have $a\in z$. From this, together with the fact that $-a\in y$, we conclude by the definition of $\perp_A$ that $z\perp_Ay$, as desired.

 We now proceed by verifying that the topology of $\mathscr{S}_0(\mathfrak{A})$ described in Definition 3.4 makes $\mathscr{S}_0(\mathfrak{A})$ an orthomodular space in the sense of Definition 3.2.   

1.) Parts 1 through 9 of this proof have shown that $\langle\mathfrak{F}(\mathfrak{A});\preceq,\perp_A,\mathfrak{P}(\mathfrak{A})\rangle$ is an orthomodular frame whenever $\mathfrak{A}$ is an orthomodular lattice. and hence condition 1 of Definition 3.2 is satisfied.

2.) The compactness of $\mathscr{S}_0(\mathfrak{A})$ is proved in the same way as Goldblatt in \cite[Proposition 3]{goldblatt1} as the proof runs the same with arbitrary (possibly improper) filters. Let $B,C\subseteq A$ be subsets of the carrier set $A$ of $\mathfrak{A}$ and let: \[\mathcal{S}_1=\{h(b):b\in B\}\cup\{\mathfrak{F}(\mathfrak{A})\setminus h(c):c\in C\}\] be a subbasic open cover of $\mathscr{S}_0(\mathfrak{A})$. By Alexander's Subbase Theorem (i.e., Theorem 3.3), it suffices to prove that $\mathcal{S}_1$ has a finite subcover. For the sake of contradiction, assume that $\mathcal{S}_1$ has no finite subcover. Then for $c_1,\dots,c_n\in C$ and $b\in B$, there exists some filter $x\in\mathfrak{F}(\mathfrak{A})$ such that:
\begin{equation*}
    x\not\in\big(\bigcup_{i\leq n}\mathfrak{F}(\mathfrak{A})\setminus h(c_i)\big)\cup h(b).\hspace{1cm}(*)
\end{equation*}
Notice that when $x=\omega$, i.e., $0\in x$, we immediately have a contradiction as $b,c_i\in\omega$ for any $i\in I$, by the definition of $\omega$. If $x$ is a proper filter i.e., $x\not=\omega$, then $c_1,\dots, c_n\in x$ but $b\not\in x$. Since $x$ is a filter, it immediately follows that $\bigwedge_{i=1}^nc_i\not\leq b$ where $\bigwedge^n_{i=1}c_i:=\inf\{c_1,\dots,c_n\}.$ Now define: \[x_C:=\{d\in A:\exists c_1\dots\exists c_n\in C . \bigwedge_{i=1}^nc_i\leq d\}\] as the filter generated by $C$. Then clearly $C\subseteq x_C$ and hence, $x_C\in\bigcap_{c\in C}h(c)$. Moreover, $x_C$ and $B$ are disjoint since $\bigwedge_{i=1}^nc_i\not\leq b$ and therefore, $x_C\not\in\bigcup_{b\in B}h(b)$. This, together with $(*)$ then implies that: \[x\not\in\Big(\bigcup_{b\in B}h(b)\Big)\cup\Big(\bigcup_{c\in C}\mathfrak{F}(\mathfrak{A})\setminus h(c)\Big)\] which contradicts the fact that the subbasis $\mathcal{S}_1$ is a cover of $\mathscr{S}_0(\mathfrak{A})$. Hence, $\mathcal{S}_1$ has a finite subcover, and by Alexander's Subbase Theorem, $\mathscr{S}_0(\mathfrak{A})$ is a compact space.        

3.) We now show that $h(a)\in\mathcal{CO}(\mathscr{S}_0(\mathfrak{A}))^{\dagger}$ for every $a\in A$. Notice that for any element $a\in A$, its $h$-image is a subbasic open set of $\mathscr{S}_0(\mathfrak{A})$ and hence $h(a)\in \mathcal{O}(\mathscr{S}_0(\mathfrak{A}))$. The definition of the subbasis $\mathcal{S}$ of $\mathcal{T}(\mathcal{S})$ on $\mathscr{S}_0(\mathfrak{A})$ then ensures that $h(a)\in \mathcal{CO}(\mathscr{S}_0(\mathfrak{A}))$ and by Proposition 5 in \cite{hartonas}, it was shown that $h(a)$ is $\perp$-stable so $h(a)\in\mathcal{CO}(\mathscr{S}_0(\mathfrak{A}))^{\dagger}$. The proof that each $U\in\mathcal{CO}(\mathscr{S}_0(\mathfrak{A}))^{\dagger}$ is of the form $h(a)$ for some $a\in A$ follows from \cite[Lemma 3.4]{bimbo} and hence $\{h(a):a\in A\}=\mathcal{CO}(\mathscr{S}_0(\mathfrak{A}))^{\dagger}$.  The closure of $\mathcal{CO}(\mathscr{S}_0(\mathfrak{A}))^{\dagger}$ under $\cap$ and $^{\perp}$ follows directly from the definition of $\omega$ along with \cite[Lemma 3.3]{bimbo}. This establishes that condition 3 holds.

4.) For condition 4, we want to show that each $U\in\mathcal{CO}(\mathscr{S}_0(\mathfrak{A}))^{\dagger}$ is of the form ${\uparrow}z$ for some principal filter $z\in\mathfrak{P}(\mathfrak{A})$, i.e., $h(a)={\uparrow}z$. Hence, given an arbitrary filter $x\in\mathfrak{F}(\mathfrak{A})$, choose any element $a\in A$ and let $z={\uparrow}a\in\mathfrak{P}(\mathfrak{A})$. Clearly if $x\in h(a)$, then $a\in x$ and hence for any $b\in z$, we have $a\leq b$ since $z={\uparrow}a$ and thus $b\in x$ since $x$ is a filter. Therefore $z\subseteq x$ and hence $x\in{\uparrow}z$. Conversely, if $x\in{\uparrow}z$, then $z={\uparrow}a\subseteq x$. Since $a\in{\uparrow}a$, we then have $a\in x$ so $x\in h(a)$. Hence $h(a)={\uparrow}z$ for some $z\in\mathfrak{P}(\mathfrak{A})$ i.e., when $z={\uparrow}a$. Since $\{h(a):a\in A\}=\mathcal{CO}(\mathscr{S}_0(\mathfrak{A}))^{\dagger}$, condition 4 holds.

5.) For condition 5, let $x,y\in\mathfrak{F}(\mathfrak{A})$ and assume $x\not\subseteq y$. Then there exists $a\in A$ such that $a\in x$ and $a\not\in y$, and hence $x\in h(a)$ and $y\not\in h(a)$ where $h(a)=U$ for some $U\in\mathcal{CO}(\mathscr{S}_0(\mathfrak{A}))^{\dagger}$.

6.) For condition 6, assume that $x\perp_Ay$. Then there exists $a\in A$ such that $a\in x$ and $-a\in y$, 
hence $x\in h(a)$ and $y\in h(-a)$. Note that $y\in h(-a)$ iff $-a\in y$ by the definition of $h$, iff $y\perp_Ax$ for each $x\in h(a)$ by the definition of $\perp_A$, iff $y\in h(a)^{\perp}$.  

Therefore, we conclude that $\mathscr{S}_0(\mathfrak{A})$ is an orthomodular space whenever $\mathfrak{A}$ is an orthomodular lattice.  
\end{proof}

\begin{definition}
Let $\mathcal{X}=\langle X;\mathcal{T}_X\rangle$ be a topological space. Then,
\begin{enumerate}
    \item $\mathcal{X}$ is a \emph{totally disconnected space} if for all points $x,y\in X$ such that $x\not=y$, there exists a clopen set $U\in\mathcal{CO}(\mathcal{X})$ such that $x\in U$ and $y\not\in U$; 
    \item $\mathcal{X}$ is a \emph{Stone space} if $\mathcal{X}$ is a compact totally disconnected space. 
\end{enumerate}
\end{definition}
By Lemma \ref{lattice to space}, we immediately have the following. 
\begin{corollary}
If $\mathfrak{A}$ is an orthomodular lattice, then $\mathscr{S}_0(\mathfrak{A})$ is a Stone space. Moreover, $\perp_A$ is an orthogonality relation on $\mathfrak{F}(\mathfrak{A})\setminus\{\omega\}$, that is, $\perp_A$ is irreflexive and symmetric on the set of proper filters in $\mathfrak{A}$.  
\end{corollary}
\begin{proof}
We have already seen in the proof of Lemma \ref{lattice to space} that $\mathscr{S}_0(\mathfrak{A})$ is a compact space. Since $\mathscr{S}_0(\mathfrak{A})$ is an orthomodular space, condition 5 of orthomodular spaces ensures that $\mathscr{S}_0(\mathfrak{A})$ is totally disconnected. 

To see that $\perp_A$ is irreflexive on $\mathfrak{F}(\mathfrak{A})\setminus\{\omega\}$, let $x\in\mathfrak{F}(\mathfrak{A})\setminus\{\omega\}$. Then if $x\perp_Ax$, there exists $a\in A$ such that $a,-a\in x$ by the definition of $\perp_A$. Then $a\cdot -a\in x$ since $x$ is a filter. Then $0\in x$ since $a\cdot -a=0$ so $x=\omega$, but this contradicts our hypothesis that $x\in\mathfrak{F}(\mathfrak{A})\setminus\{\omega\}$. Therefore $x\not\perp_Ax$. The symmetry of $\perp_A$ is guaranteed by condition 8 of orthomodular frames.  
\end{proof}
\begin{lemma}[Dual lattice of an orthomodular space]\label{space to lattice}
If $\mathcal{X}=\langle X;\preceq,\perp,\Omega,\mathcal{T}_X\rangle$ is an orthomodular space, then the algebra $\mathscr{A}_0(\mathcal{X})=\langle\mathcal{CO}(\mathcal{X})^{\dagger};\cap, \hspace{.015cm}^{\perp},\{\omega\}\rangle$ is an orthomodular lattice. Hence, $\mathscr{A}_0(\mathcal{X})$ denotes an orthomodular lattice whenever $\mathcal{X}$ is an orthomodular space.
\end{lemma}
\begin{proof}
The proof is similar to the soundness theorem for orthomodular quantum logic in \cite{hartonas}. We first show that $\mathscr{A}_0(\mathcal{X})$ is an ortholattice in the sense of Definition 2.5.

1.) It is obvious that $\mathscr{A}_0(\mathcal{X})$ is a meet-semilattice as $\cap$ is clearly associative, commutative, and idempotent with respect to $\mathcal{CO}(\mathcal{X})^{\dagger}$. To see that $\{\omega\}$ is the least element of $\mathscr{A}_0(\mathcal{X})$, it suffices to show that $\{\omega\}\subseteq U$ for each $U\in\mathcal{CO}(\mathcal{X})^{\dagger}$, i.e., $\{\omega\}=\{\omega\}\cap U$. For the left-to-right inclusion, if $x\in\{\omega\}$, then $x=\omega$. By condition 4 of Definition 3.2, $U$ is a principal upset with respect to $\preceq$ and by condition 2 of Definition 2.10, $\omega$ is the top element of $\langle X;\preceq\rangle$, and hence $\omega\in U$. Therefore, it is obvious that $\omega\in\{\omega\}\cap U$. The right-to-left inclusion is trivial. Hence $\{\omega\}=\{\omega\}\cap U$ for each $U\in\mathcal{CO}(\mathcal{X})^{\dagger}$ and thus $\{\omega\}\subseteq U$ for each $U\in\mathcal{CO}(\mathcal{X})^{\dagger}$ which means that $\{\omega\}$ is the least element in $\mathscr{A}_0(\mathcal{X})$.      

  2.a.) To see that $U\cap U^{\perp}=\{\omega\}$ for any $U\in\mathcal{CO}(\mathcal{X})^{\dagger}$, first note that if $x\in U\cap U^{\perp}$, then $x\in U$ and $x\in U^{\perp}$ so $x\perp U$ and furthermore $x\perp x$ since $x\in U$. By hypothesis, $\mathcal{X}$ is an orthomodular space and hence $\langle X;\preceq,\perp,\Omega\rangle$ is an orthomodular frame. Thus by condition 5 of orthomodular frames, we have $x=\omega$ so $x\in\{\omega\}$.

The converse inclusion $\{\omega\}\subseteq U\cap U^{\perp}$ follows immediately from the fact that $U$ and $U^{\perp}$ are principal upsets (by condition 4 of Definition 3.2) and that $\omega$ is the top element in the meet-semilattice $\langle X;\preceq\rangle$.

  2.b.) To see that $^\perp$ is order reversing on $\mathcal{CO}(\mathcal{X})^{\dagger}$, let $U,V\in\mathcal{CO}(\mathcal{X})^{\dagger}$ and assume that $U\subseteq V$ and $x\in V^{\perp}$. There are two cases: First, if $x=\omega$, then $x\perp y$ for all $y\in U$ since $\omega\perp y$ for all $y\in X$ by condition 7 of orthomodular frames. Hence $x\in U^{\perp}$ by the definition of $^{\perp}$. Second, if $x\not=\omega$, then $x\perp y$ for all $y\in V$ by our hypothesis that $x\in V^{\perp}$ so $x\not\in V$ since $x\not=\omega$. Also $x\not\in U$ since $U\subseteq V$ by hypothesis. Then $x\in U^{\perp}$ since $x\not\in U$ and $x\not=\omega$. Notice that in general we have $U^{\perp}=\{X\setminus U\}\cup\{\omega\}$. We conclude that $\mathcal{CO}(\mathcal{X})^{\dagger}$ is an ortholattice for any orthomodular space $\mathcal{X}$.

    2.c.) It is trivial that $^{\perp}$ is an involution on $\mathcal{CO}(\mathcal{X})^{\dagger}$ since $(\mathcal{X})^{\dagger}:=\{U\in\wp(X):U=U^{\perp\perp}\}$ by definition, and hence $U=U^{\perp\perp}$ for any $U\in\mathcal{CO}(\mathcal{X})^{\dagger}$.

Parts 1 through 2.c of this proof establish that $\mathscr{A}_0(\mathcal{X})$ is an ortholattice whenever $\mathcal{X}$ is an orthomodular space.

Condition 2, Proposition 2.8.) To see that the ortholattice $\langle\mathcal{CO}(\mathcal{X})^{\dagger};\cap,\hspace{.015cm}^{\perp},\{\omega\}\rangle$ is indeed orthomodular, by Proposition \ref{oml}, it suffices to show: \[U\subseteq V\wedge U^{\perp}\cap V=\{\omega\}\hspace{.1cm}\Rightarrow\hspace{.1cm} U=V\] for all $U,V\in\mathcal{CO}(\mathcal{X})^{\dagger}$. Observe that by condition 4 of orthomodular spaces, each $U\in\mathcal{CO}(\mathcal{X})^{\dagger}$ is of the form ${\uparrow}z$ for some point $z\in\Omega$. Hence let $U={\uparrow}x$ and $V={\uparrow}y$ for some $x,y\in\Omega$. It therefore suffices to show: \[{\uparrow}x\subseteq{\uparrow}y\wedge ({\uparrow}x)^{\perp}\cap{\uparrow}y=\{\omega\}\hspace{.1cm}\Rightarrow\hspace{.1cm}{\uparrow}x={\uparrow}y.\] Let us assume the hypothesis and note that ${\uparrow}x\subseteq{\uparrow}y$ iff $y\preceq x$, thus ${\uparrow}x={\uparrow}y$ iff $x=y$. Moreover, observe that by the definition of $^\perp$ and condition 9 of orthomodular frames, we have $z\in({\uparrow}x)^{\perp}$ iff $z\perp x$ and by the definition of $\uparrow$, we also have $z\in{\uparrow}y$ iff $y\preceq z$. Hence it is easy to see that our hypothesis that ${\uparrow}x\subseteq{\uparrow}y$ and $({\uparrow}x)^{\perp}\cap {\uparrow}y=\{\omega\}$ is equivalent to the requirement that for all $x,y\in\Omega$, if $y\preceq x$ and for all $z\in X$, if $y\preceq z$ and $z\perp x$, then $z=\omega$. Thus by condition 4 of orthomodular frames, we have $x=y$ and therefore ${\uparrow}x={\uparrow}y$.

This completes the proof that $\mathscr{A}_0(\mathcal{X})$ is an orthomodular lattice whenever $\mathcal{X}$ is an orthomodular space. 
\end{proof}
\begin{theorem}[Topological representation of orthomodular lattices]\label{representation} For every orthomodular lattice $\mathfrak{A}=\langle A;\cdot,-,0\rangle$, there exists an isomorphism $\mathfrak{A}\to\mathscr{A}_0(\mathscr{S}_0(\mathfrak{A}))$. 
\end{theorem}
\begin{proof}
We prove that the canonical map $h\colon \mathfrak{A}\to\wp(\mathfrak{F}(\mathfrak{A}))$ as defined in Definition 3.4 exhibits the desired isomorphism from $\mathfrak{A}$ to $\mathscr{A}_0(\mathscr{S}_0(\mathfrak{A}))$. To see that $h$ is a homomorphism, there are three cases.

We first prove that $h(a\cdot b)=h(a)\cap h(b)$: 
\begin{align*}
    h(a\cdot b)&=\{x\in\mathfrak{F}(\mathfrak{A}):a\cdot b\in x\}\tag{by the definition of $h$} 
    \\&=\{x\in\mathfrak{F}(\mathfrak{A}):a\in x\wedge b\in x\}\tag{since $x$ is a filter}
    \\&=\{x\in\mathfrak{F}(\mathfrak{A}):a\in x\}\cap\{x\in\mathfrak{F}(\mathfrak{A}):b\in x\}\tag{by the definition of $\cap$}
    \\&=h(a)\cap h(b)\tag{by the definition of $h$}
\end{align*}
We now demonstrate that $h(-a)=h(a)^{\perp}$: 
\begin{align*}
    h(-a)&=\{x\in\mathfrak{F}(\mathfrak{A}):-a\in x\}\tag{by the definition of $h$}
    \\&=\{x\in \mathfrak{F}(\mathfrak{A}):\forall y\in h(a).\hspace{.05cm}x\perp_Ay\}\tag{by the definition of $\perp_A$}
    \\&=h(a)^{\perp}\tag{by the definition of $^{\perp}$}
\end{align*}
Finally, we show that $h(0)=\{\omega\}$:
 \begin{align*}
     h(0)&=\{x\in\mathfrak{F}(\mathfrak{A}):0\in x\}\tag{by the definition of $h$}
     \\&=\{x\in\mathfrak{F}(\mathfrak{A}):x=\omega\}\tag{by the definition of $\omega$}
     \\&=\{\omega\}\tag{since $\omega$ is unique}
 \end{align*}

 It remains to prove that $h$ is a bijection. To see that $h$ is an injection, assume $a\not=b$. Without loss of generality, we have $b\not\leq a$. Hence $a\not\in\uparrow\hspace{-.1cm}b$ but clearly $b\in\uparrow\hspace{-.1cm}b$ so $h(a)\not=h(b)$. Recall that from Lemma 3.5, it was shown that each $U\in\mathcal{CO}(\mathfrak{F}(\mathfrak{A}))^{\dagger}$ is of the form $h(a)$ for some $a\in A$. This implies that $h$ is surjective. 

Since $h$ is a bijective homomorphism from $\mathfrak{A}$ to $\mathscr{A}_0(\mathscr{S}_0(\mathfrak{A}))$, $h$ is an isomorphism.   
\end{proof}

The following result is well-known and will be used in the proof of Theorem 3.11. 
\begin{lemma}
Let $\mathcal{X}$ be a compact Hausdorff space and let $\mathcal{X}'$ be any Hausdorff space. If $f\colon\mathcal{X}\to\mathcal{X}'$ is a continuous bijection, then $f$ is a homeomorphism. 
\end{lemma}
\begin{proof}
See Davey and Priestley \cite[Lemma A.7, p. 277]{davey}. 
\end{proof}
Recall that if $\mathcal{X}=\langle X; R_X,Q_X,\mathcal{T}_X\rangle$ and $\mathcal{Y}=\langle Y; R_Y,Q_Y, \mathcal{T}_Y\rangle$ are relational topological spaces in which $R_X,Q_X\subseteq X\times X$ and $R_Y,Q_Y\subseteq Y\times Y$, then a continuous function $f\colon \mathcal{X}\to \mathcal{Y}$ is a \emph{relational homeomorphism} if $f$ is a homeomorphism and a relational isomorphism i.e., for all $x,y\in X$, we have $xR_Xy\Leftrightarrow f(x)R_Yf(y)$ and $xQ_Xy\Leftrightarrow f(x)Q_Yf(y)$. 
  
\begin{theorem}[Algebraic realization of orthomodular spaces]\label{realization}
For an arbitrary orthomodular space $\mathcal{X}=\langle X;\preceq,\perp, \Omega,\mathcal{T}_X\rangle$, there exists a relational homeomorphism $\mathcal{X}\to\mathscr{S}_0(\mathscr{A}_0(\mathcal{X}))$. 
\end{theorem}
\begin{proof}
We demonstrate that the function $f\colon \mathcal{X}\to\mathfrak{F}(\mathcal{CO}(\mathcal{X})^{\dagger})$ defined by:  \[f(x)=\{U\in\mathcal{CO}(\mathcal{X})^{\dagger}:x\in U\}\] exhibits the desired relational homeomorphism from $\mathcal{X}$ to $\mathscr{S}_0(\mathscr{A}_0(\mathcal{X}))$.

Clearly $f$ is a total function in virtue of its definition via set abstraction, i.e., for any $x\in X$, the value of $f(x)$ is a set. We now check that $f(x)$ is a filter. Let $U,V\in\mathcal{CO}(\mathcal{X})^{\dagger}$ and assume that $U\subseteq V$. If $U\in f(x)$, then $x\in U$ by the definition of $f$. Then $x\in V$ since $U\subseteq V$ by hypothesis, and hence $V\in f(x)$, by the definition of $f$. Moreover, if $U,V\in f(x)$, then $x\in U$ and $x\in V$ by the definition of $f$ and hence $x\in U\cap V$. By hypothesis, $U$ and $V$ are clopen $\perp$-stable sets and hence by condition 3 of Definition 3.2, $U\cap V$ is also a clopen $\perp$-stable set, and thus $U\cap V\in f(x)$ by the definition of $f$. Therefore, we can conclude that $f(x)$ is a filter.

     For injectivity, let $x,y\in X$ and assume $x\not=y$. If $x\not\preceq y$, then by condition 5 of orthomodular spaces, there exists $U\in\mathcal{CO}(\mathcal{X})^{\dagger}$ such that $x\in U$ and $y\not\in U$. By the definition of $f$, we have $U\in f(x)$ but $U\not\in f(y)$ and hence $f(x)\not=f(y)$. Similarly, if $y\not\preceq x$, then there exists $U\in\mathcal{CO}(\mathcal{X})^{\dagger}$ such that $y\in U$ but $x\not\in U$ and hence $f(x)\not= f(y)$.

     To see that $f$ is a surjective function, let $z\in\mathfrak{F}(\mathcal{CO}(\mathcal{X})^{\dagger})$ be the improper filter in $\mathcal{CO}(\mathcal{X})^{\dagger}$ so that $z=X$. Since each $U\in\mathcal{CO}(\mathcal{X})^{\dagger}$ is a principal upset and $\omega$ is the top element in $\langle X;\preceq\rangle$, we calculate:
     \[f(\omega)=\{U\in\mathcal{CO}(\mathcal{X})^{\dagger}:\omega\in U\}=X=z.\] The proof that $z=f(x)$ for some $x\in X$ where $x\not=\omega$ whenever $z\in\mathfrak{F}(\mathcal{CO}(\mathcal{X})^{\dagger})$ is a proper filter runs similarly to \cite[Theorem 5, p. 520--522]{iturrioz}. Hence $f$ is a surjection and since we have already seen that $f$ is injective, we have that $f$ is a bijection.

     To see that $f$ is a continuous function, we first show that $f^{-1}[h(U)]\in\mathcal{O}(\mathcal{X})$ for each open set $h(U)\in\mathcal{O}(\mathscr{S}_0(\mathscr{A}_0(\mathcal{X})))$. The result is achieved by observing the following calculation:
\begin{align*}
    f^{-1}[h(U)]&=\{x\in X:f(x)\in h(U)\}\tag{by the definition of $f^{-1}$}
    \\&=\{x\in X:U\in f(x)\}\tag{by the definition of $h$}
    \\&=\{x\in X:x\in U\}\tag{by the definition of $f$}
    \\&=U
\end{align*}
Similarly, we now verify that $f^{-1}[\mathfrak{F}(\mathcal{CO}(\mathcal{X})^{\dagger})\setminus h(U)]\in\mathcal{O}(\mathcal{X})$ for each open set of the form  $\mathfrak{F}(\mathcal{CO}(\mathcal{X})^{\dagger})\setminus h(U)\in\mathscr{S}_0(\mathscr{A}_0(\mathcal{X}))$: 
\begin{align*}
    f^{-1}[\mathfrak{F}(\mathcal{CO}(\mathcal{X})^{\dagger})\setminus h(U)]&=\{x\in X:f(x)\in\mathfrak{F}(\mathcal{CO}(\mathcal{X})^{\dagger})\setminus h(U)\}\tag{by the definition of $f^{-1}$}
    \\&=\{x\in X:f(x)\not\in h(U)\}\tag{by the definition of $\setminus$}
    \\&=\{x\in X:U\not\in f(x)\}\tag{by the definition of $h$}
    \\&=\{x\in X:x\not\in U\}\tag{by the definition of $f$}
    \\&=X\setminus U\tag{by the definition of $\setminus$}
\end{align*}

 Having already noted that every orthomodular space is a totally disconnected space, it is easy to see that every orthomodular space is also a Hausdorff space. Then, since $f$ is a continuous bijection from a compact Hausdorff space to a Hausdorff space, it immediately follows from Lemma 3.10 that $f$ is a homeomorphism.

 To see that in addition, $f$ is a relational homeomorphism, we first show that $f$ is a relational isomorphism with respect to the orthomodular frame reducts $\langle X;\perp_X\rangle$ and $\langle \mathfrak{F}(\mathcal{CO}(\mathcal{X})^{\dagger});\perp_{\mathfrak{F}(\mathcal{CO}(\mathcal{X})^{\dagger})}\rangle$. The result is achieved by observing that for arbitrary points $x,y\in X$, we have:
\begin{align*}
    x\perp_X y &\Leftrightarrow \exists U\in\mathcal{CO}(\mathcal{X})^{\dagger}[x\in U\wedge y\in U^{\perp}]\tag{by the definition of $\perp_X$}
    \\&\Leftrightarrow \exists U\in\mathcal{CO}(\mathcal{X})^{\dagger}[U\in f(x)\wedge U^{\perp}\in f(y)]\tag{by the definition of $f$}
    \\&\Leftrightarrow f(x)\perp_{\mathfrak{F}(\mathcal{CO}(\mathcal{X})^{\dagger})}f(y)\tag{by the definition of $\perp_{\mathfrak{F}(\mathcal{CO}(\mathcal{X})^{\dagger})}$}
\end{align*}

 Lastly, we verify that $f$ is a relational isomorphism with respect to the partial order reducts $\langle X,\preceq\rangle$ and $\langle\mathfrak{F}(\mathcal{CO}(\mathcal{X})^{\dagger});\subseteq\rangle$. The result is almost immediate by noting:
 \begin{align*}
     x\not\preceq_Xy&\Leftrightarrow\exists U\in\mathcal{CO}(\mathcal{X})^{\dagger}[x\in U\wedge y\not\in U]\tag{by definition 3.2}
     \\&\Leftrightarrow \exists U\in\mathcal{CO}(\mathcal{X})^{\dagger}[U\in f(x)\wedge U\not\in f(y)]\tag{by the definition of $f$}
     \\&\Leftrightarrow f(x)\not\subseteq f(y)\tag{by the definition of $\subseteq$}
 \end{align*}
 
Hence $f$ is a relational homeomorphism from $\mathcal{X}$ to $\mathscr{S}_0(\mathscr{A}_0(\mathcal{X}))$ for any orthomodular space $\mathcal{X}$.
\end{proof}

\section{Duality Theory}

Within thi section, it is assumed that the reader is familiar with basic notions in category theory such as contravariant functor and dual equivalence of categories. We refer to Awodey in \cite{awo} for detailed treatment of basic category theory.

\begin{definition}
The category $\mathbb{OML}$ has as objects orthomodular lattices, and has as morphisms \emph{orthomodular lattice homomorphisms}, i.e., homomorphisms $\phi\colon \mathfrak{A}\to \mathfrak{A}'$ which preserve the orthomodular lattice operations: 

\begin{enumerate}
    \item $\phi(a\cdot b)=\phi(a)\cdot\phi(b)$
    \item $\phi(-a)=-\phi(a)$
    \item $\phi(0)=0'$
\end{enumerate}

\end{definition}
The following class of continuous frame morphisms were introduced in \cite{bimbo} and will accompany the objects in the category of orthomodular spaces. 
\begin{definition}\label{p-morphism}
The category $\mathbb{OMS}$ has as objects orthomodular spaces and has as morphisms \emph{continuous weak p-morphisms}, i.e., continuous functions $\psi\colon\mathcal{X}\to\mathcal{X}'$ satisfying the following conditions: 
 \begin{enumerate}
    \item $x\not\perp_Xy\Rightarrow \psi(x)\not\perp_{X'}\psi(y)$
    \item $z'\not\perp_{X'}\psi (y)\Rightarrow\exists x(x\not\perp_Xy\wedge z'\preceq_{X'}\psi(x))$
\end{enumerate}
\end{definition}

The frame conditions within Definition 4.2 are meant to resemble those of a $p$-morphism in the sense of modal logic with respect to $\not\perp$, and hence we do not need to include conditions such as boundedness or monotonicity with respect to $\preceq$ (see for instance Blackburn, de Rijke, and Venema in \cite[Chapter 3]{blackburn} for more details on $p$-morphisms). Observe that if we replace $\preceq_{X'}$ by $=$ in the right-hand conjunct of the consequent of condition 2, then the conditions on the frame correspond exactly to that of a $p$-morphism. Notice that over $X\setminus\{\omega\}$ and $X'\setminus\{\omega'\}$, the resulting $p$-morphism is that of a $p$-morphism on the modal $B$-frames of reflexive symmetric modal logic. See \cite{goldblatt2} for more details on the normal modal logic $B$.

\begin{lemma}[The contravariant functor $\mathscr{S}_F$]\label{homomorphism to continuous map}
Let $\mathfrak{A}$ and $\mathfrak{A}'$ be orthomodular lattices and let $\phi\colon \mathfrak{A}\to\mathfrak{A}'$ an orthomodular lattice homomorphism. Then the map $\mathscr{S}_1(\phi)\colon\mathscr{S}_0(\mathfrak{A}')\to\mathscr{S}_0(\mathfrak{A})$ defined by $\mathscr{S}_1(\phi):=\phi^{-1}$ is a continuous weak $p$-morphism. Moreover, $\mathscr{S}_F:=\langle\mathscr{S}_0,\mathscr{S}_1\rangle\colon\mathbb{OML}\to\mathbb{OMS}$ is a contravariant functor.  
\end{lemma}
\begin{proof}
Given our hypothesis, it suffices to show that $\phi^{-1}$ is a continuous weak $p$-morphism from $\mathscr{S}_0(\mathfrak{A}')$ to $\mathscr{S}_0(\mathfrak{A})$ whenever $\phi\colon\mathfrak{A}\to\mathfrak{A}'$ is an orthomodular lattice homomorphism.

    We show that $\phi^{-1}$ is a continuous function by proving $(\phi^{-1})^{-1}[U]$ is open in $\mathscr{S}_0(\mathfrak{A}')$ whenever $U$ is a subbasic open set in $\mathscr{S}_0(\mathfrak{A})$. There are two cases: 
    
    The calculation for the case when $U$ is of the form $h(a)$ for some $a\in A$ runs as follows: 
\begin{align*}
    x\in(\phi^{-1})^{-1}[h(a)]&\Leftrightarrow \phi^{-1}[x]\in h(a)\tag{by the definition of $\phi^{-1}$}
    \\&\Leftrightarrow a\in\phi^{-1}[x]\tag{by the definition of $h$}
    \\&\Leftrightarrow \phi(a)\in x\tag{by the definition of $\phi^{-1}$}
    \\&\Leftrightarrow x\in h(\phi(a))\tag{by the definition of $h$}
\end{align*}
The case for when $U$ is of the form $\mathfrak{F}(\mathfrak{A})\setminus h(a)$ for some $a\in A$ runs similarly: 
 \begin{align*}
     x\in(\phi^{-1})^{-1}[\mathfrak{F}(\mathfrak{A})\setminus h(a)]&\Leftrightarrow\phi^{-1}[x]\in\mathfrak{F}(\mathfrak{A})\setminus h(a)\tag{by the definition of $\phi^{-1}$}
     \\&\Leftrightarrow\phi^{-1}[x]\not\in h(a)\tag{by the definition of $\setminus$}
    \\&\Leftrightarrow a\not\in\phi^{-1}[x]\tag{by the definition of $h$}
    \\&\Leftrightarrow\phi(a)\not\in x\tag{by the definition of $\phi^{-1}$}
    \\&\Leftrightarrow x\not\in h(\phi(a))\tag{by the definition of $h$}
    \\&\Leftrightarrow x\in\mathfrak{F}(\mathfrak{A})\setminus h(\phi(a))\tag{by the definition of $\setminus$}
 \end{align*}

     To see that $\phi^{-1}$ is in addition a weak $p$-morphism, first notice that conditions 1 and 2 of Definition \ref{p-morphism} are trivially satisfied when either argument place of the antecedent is $\omega$ since $\omega\perp_A x$ for any $x\in\mathscr{S}_0(\mathfrak{A})$ and $\omega\perp_{A'} y$ for any $y\in\mathscr{S}_0(\mathfrak{A}')$. Now let $x,y\in\mathscr{S}_0(\mathfrak{A}')$ such that $x,y\not=\omega$. For condition 1 of Definition \ref{p-morphism}, assume for the sake of contraposition that $\phi^{-1}[x]\perp_{A}\phi^{-1}[y]$ so that there exists $a\in A$ such that $a\in \phi^{-1}[x]$ and $-a\in\phi^{-1}[y]$. Therefore $\phi(a)\in x$ and $\phi(-a)=-\phi(a)\in y$ and hence $x\perp_{A'}y$.

 For condition 2 of Definition \ref{p-morphism}, assume that $z\not\perp_{A}\phi^{-1}[y]$. Since $y\not=\omega$ by hypothesis, it is immediate that there exists $x\in\mathscr{S}_0(\mathfrak{A}')$ such that $x\not\perp_{A'}y$. The monotonicity of $^{-1}$ ensures that the right-hand conjunct in the consequent of condition 2 is satisfied.

 It remains to verify that $\mathscr{S}_F$ is indeed a contravariant functor. First recall that Lemma 3.5 showed that if $\mathfrak{A}$ is an orthomodular lattice, then $\mathscr{S}_0(\mathfrak{A})$ is an orthomodular space. This establishes that $\mathscr{S}_0$ is a functor on objects. We now check that $\mathscr{S}_1$ is a functor on morphisms. Hence, let $\mathfrak{A}_1$, $\mathfrak{A}_2$, and $\mathfrak{A}_3$ be orthomodular lattices, let $\psi_1\colon\mathfrak{A}_1\to\mathfrak{A}_2$ and $\phi_2\colon\mathfrak{A}_2\to\mathfrak{A}_3$ be homomorphisms, and define their composition $\phi_2\circ\phi_1\colon\mathfrak{A}_1\to\mathfrak{A}_3$ by $\phi_2\circ\phi_1(a):=\phi_2(\phi_1(a))$ for all $a\in A_1$. The first part of this proof established that if $\phi\colon\mathfrak{A}\to\mathfrak{A}'$ is a homomorphism of orthomodular lattices, then $\mathscr{S}_1(\phi)\colon\mathscr{S}_0(\mathfrak{A}')\to\mathscr{S}_0(\mathfrak{A})$ is a continuous weak $p$-morphism between orthomodular spaces. This, together with the fact that continuous functions and $p$-morphisms (and hence weak $p$-morphisms) are closed under composition, implies that $\mathscr{S}_1(\phi)\colon\mathscr{S}_0(\mathfrak{A}_3)\to\mathscr{S}_0(\mathfrak{A}_1)$ is a continuous weak $p$-morphism whenever $\phi_2\circ\phi_1\colon\mathfrak{A}_1\to\mathfrak{A}_3$ is a homomorphism. Moreover, we have $\mathscr{S}_1(\phi_2\circ\phi_1)=\mathscr{S}_1(\phi_1)\circ\mathscr{S}_0(\phi_2)$ since $(\phi_2\circ\phi_1)^{-1}=\phi_1^{-1}\circ\phi_2^{-1}$. Lastly, let $\textit{id}_{\mathfrak{A}}\colon\mathfrak{A}\to\mathfrak{A}$ be the trivial identity homomorphism on an orthomodular lattice $\mathfrak{A}$ defined by $\textit{id}_{\mathfrak{A}}(a)=a$ for all $a\in A$. Clearly we have $\mathscr{S}_1(\textit{id}_{\mathfrak{A}})=\textit{id}_{\mathscr{S}_0(\mathfrak{A})}$ since $\textit{id}_{\mathfrak{A}}^{-1}=\textit{id}_{\mathfrak{A}}$. This shows that $\mathscr{S}_F:=\langle\mathscr{S}_0,\mathscr{S}_1\rangle\colon\mathbb{OML}\to\mathbb{OMS}$ is a contravariant functor.
\end{proof}

\begin{lemma}[The contravariant functor $\mathscr{A}_F$]\label{continuous map to homomorphism}
Let $\mathcal{X}$ and $\mathcal{X}'$ be orthomodular spaces and $\psi\colon \mathcal{X}\to \mathcal{X}'$ a continuous weak $p$-morphism. Then the map $\mathscr{A}_1(\psi)\colon\mathscr{A}_0(\mathcal{X}')\to\mathscr{A}_0(\mathcal{X})$ defined by $\mathscr{A}_1(\psi):=\psi^{-1}$ is an orthomodular lattice homomorphism. Moreover, $\mathscr{A}_F:=\langle\mathscr{A}_0,\mathscr{A}_1\rangle\colon\mathbb{OMS}\to\mathbb{OML}$ is a contravariant functor. 
\end{lemma}
\begin{proof}
Given our hypothesis, we want to show that $\psi^{-1}$ is an orthomodular lattice homomorphism from $\mathscr{A}_0(\mathcal{X}')$ to $\mathscr{A}_0(\mathcal{X})$ whenever $\psi\colon \mathcal{X}\to \mathcal{X}'$ is a continuous weak $p$-morphism.

     To see that $\psi^{-1}$ is a homomorphism, we first calculate $\psi^{-1}[\{\omega\}]=\{\omega\}$. For the left-to-right inclusion, assume that $y\not\in\psi^{-1}[\{\omega\}]$ for the sake of contraposition. Then $\psi(y)\not\in\{\omega\}$ by the definition of $\psi^{-1}$. Since $\omega$ is an upper bound for the meet-semilattice $\langle X';\preceq_{X'}\rangle$, it follows that $\psi(y)\prec_{X'}\omega$ and hence there exists some $z\in X'$ such that $\psi(y)\not\perp_{X'}z$. By hypothesis, $\psi$ is a continuous weak $p$-morphism and hence by condition 2 of Definition 4.2, there exists $x\in X$ such that $x\not\perp_Xy$ and $z\preceq_{X'}\psi(x)$. Since $\omega\perp_X v$ for any point $v\in X$, it follows that $y\not=\omega$ and hence $y\not\in\{\omega\}$.

    For the right-to-left inclusion, assume $x\in\psi^{-1}[\{\omega\}]$. Then $\psi(x)\in\{\omega\}$ by the definition of $\psi^{-1}$ and thus $\psi(x)=\omega$. Hence $\psi(x)\perp_{X'}y$ for each $y\in X'$. Now assume for the sake of contradiction that $x\not\in\{\omega\}$. Then $x\not=\omega$ and so there exists $z\in X$ such that $x\not\perp_Xz$. Since $\psi$ is a continuous weak $p$-morphism, by condition 1 of Definition 4.2, it follows that $\psi(x)\not\perp_{X'}\psi(z)$ but this contradicts the fact that $\psi(x)\perp_{X'} y$ for each $y\in X'$ including when $y=\psi(z)$. Hence $x\in\{\omega\}$. 
   
     The calculation that $\psi^{-1}[U\cap V]=\psi^{-1}[U]\cap\psi^{-1}[V]$ is standard:
\begin{align*}
    \psi^{-1}[U\cap V]&=\{x\in X:\psi(x)\in U\cap V\}\tag{by the definition of $\psi^{-1}$}
    \\&=\{x\in X:\psi(x)\in U\wedge\psi(x)\in V\}\tag{by the definition of $\cap$}
    \\&=\{x\in X:\psi(x)\in U\}\cap\{x\in X:\psi(x)\in V\}\tag{by the definition of $\cap$}
    \\&=\psi^{-1}[U]\cap\psi^{-1}[V]\tag{by the definition of $\psi^{-1}$}
\end{align*}
 We now show that $\psi^{-1}[U^{\perp}]=\psi^{-1}[U]^{\perp}$. For the left-to-right inclusion, assume that $x\not\in\psi^{-1}[U]^{\perp}$ for the sake of contraposition. Then there exists $y\in\psi^{-1}[U]$ such that $x\not\perp_X y$. Since $\psi$ is a continuous weak $p$-morphism, by condition 1 of Definition \ref{p-morphism}, it follows that $\psi(x)\not\perp_{X'}\psi(y)$. Then since $\psi(y)\in U$, it follows that $\psi(x)\not\in U^{\perp}$ and hence $x\not\in\psi^{-1}[U^{\perp}]$.

For the right-to-left inclusion, assume that $x\not\in\psi^{-1}[U^{\perp}]$ for the sake of contraposition. Then $\psi(x)\not\in U^{\perp}$ so there exists $y\in U$ such that $y\not\perp_{X'}\psi(x)$. By condition 2 of Definition \ref{p-morphism}, there exists $z$ such that $z\not\perp_{X'}x$ and $y\preceq_{X'}\psi(z)$. Then $\psi(z)\in U$ so that $z\in\psi^{-1}[U]$. This together with $z\not\perp_{X'}x$ implies $x\not\in\psi^{-1}[U]^{\perp}$.

 It remains to verify that $\mathscr{A}_F$ is indeed a functor. Recall that Lemma 3.8 demonstrated that if $\mathcal{X}$ is an orthomodular space, then $\mathscr{A}_0(\mathcal{X})$ is an orthomodular lattice. This establishes that $\mathscr{A}_0$ is a functor on objects. We now check that $\mathscr{A}_1$ is a functor on morphisms. Hence, let $\mathcal{X}_1$, $\mathcal{X}_2$, and $\mathcal{X}_3$ be orthomodular spaces, let $\psi_1\colon\mathcal{X}_1\to\mathcal{X}_2$ and $\psi_2\colon\mathcal{X}_2\to\mathcal{X}_3$ be continuous weak $p$-morphisms, and define their composition map $\psi_2\circ\psi_1\colon\mathcal{X}_1\to\mathcal{X}_3$ in the usual way by $\psi_2\circ \psi_1(x):=\psi_2(\psi_1(x))$ for each point $x\in X_1$. The first part of this proof established that if $\psi\colon\mathcal{X}\to\mathcal{X}'$ is a continuous weak $p$-morphism between orthomodular spaces, then $\mathscr{A}_1(\psi)\colon\mathscr{A}_0(\mathcal{X}')\to\mathscr{A}_0(\mathcal{X})$ is a homomorphism between orthomodular lattices. This, together with the fact that homomorphisms are closed under composition, implies that $\mathscr{A}_1(\psi_2\circ\psi_1)\colon\mathscr{A}_0(\mathcal{X}_3)\to\mathscr{A}_0(\mathcal{X}_1)$ is a homomorphism whenever $\psi_2\circ\psi_1\colon\mathcal{X}_1\to\mathcal{X}_3$ is a continuous weak $p$-morphism. Moreover, we have $\mathscr{A}_1(\psi_2\circ\psi_1)=\mathscr{A}_1(\psi_1)\circ\mathscr{A}_1(\psi_2)$ since $(\psi_2\circ\psi_1)^{-1}=\psi_1^{-1}\circ\psi_2^{-1}$. Lastly, notice that for any orthomodular space $\mathcal{X}$, given its trivial identity map $\textit{id}_{\mathcal{X}}\colon\mathcal{X}\to\mathcal{X}$ defined by $\textit{id}_{\mathcal{X}}(x)=x$ for any $x\in X$, we have $\mathscr{A}_1(\textit{id}_{\mathcal{X}})=\textit{id}_{\mathscr{A}_0(\mathcal{X})}$ since $\textit{id}_{\mathcal{X}}^{-1}=\textit{id}_{\mathcal{X}}$. This shows that $\mathscr{A}_F:=\langle\mathscr{A}_1,\mathscr{A}_1\rangle\colon\mathbb{OMS}\to\mathbb{OML}$ is a contravariant functor.
\end{proof}
We now proceed to the main theorem of this paper. 
\begin{theorem}[Functorial duality between $\mathbb{OML}$ and $\mathbb{OMS}$]\label{duality}
The contravariant functors $\mathscr{A}_F\colon\mathbb{OMS}\to\mathbb{OML}$ and $\mathscr{S}_F\colon\mathbb{OML}\to\mathbb{OMS}$ constitute a dual equivalence of categories. 
\end{theorem}
\begin{proof}
Let $\mathfrak{A}$ and $\mathfrak{A}'$ be orthomodular lattices, let $\phi\colon \mathfrak{A}\to \mathfrak{A}'$ be an orthomodular lattice homomorphism, and let $h\colon \mathfrak{A}\to\wp(\mathfrak{F}(\mathfrak{A}))$ be the canonical representation map defined in Definition 3.4. In Lemma 4.3, it was already shown that: 
\[h(\phi(a))=(\phi^{-1})^{-1}[h(a)]\hspace{.1cm}\text{i.e.,}\hspace{.1cm}h(\phi(a))=\mathscr{A}_1(\mathscr{S}_1(\phi))[h(a)]\]

Now let $\mathcal{X}$ and $\mathcal{X}'$ be orthomodular spaces, let $\psi\colon \mathcal{X}\to \mathcal{X}'$ be a continuous weak $p$-morphism and let $f\colon \mathcal{X}\to\mathfrak{F}(\mathcal{CO}(\mathcal{X})^{\dagger})$ be the algebraic realization map defined in the statement of Theorem \ref{realization}. We want to check that for each point $x\in X$, 
\[f(\psi(x))=(\psi^{-1})^{-1}[f(x)]\hspace{.1cm}\text{i.e.,}\hspace{.1cm}f(\psi(x))=\mathscr{S}_1(\mathscr{A}_1(\psi))[f(x)].\]
The result falls out of the definitions of $f$ and $\psi^{-1}$:
\begin{align*}
    U\in f(\psi(x))&\Leftrightarrow \psi(x)\in U\tag{by the definition of $f$}
    \\&\Leftrightarrow x\in\psi^{-1}[U]\tag{by the definition of $\psi^{-1}$}
    \\&\Leftrightarrow \psi^{-1}[U]\in f(x)\tag{by the definition of $f$}
    \\&\Leftrightarrow U\in(\psi^{-1})^{-1}[f(x)]\tag{by the definition of $\psi^{-1}$}
\end{align*}
 This, together with the representation results achieved in section 3 along with Lemmas \ref{homomorphism to continuous map} and \ref{continuous map to homomorphism}, implies that the following diagrams commute:

\begin{tikzcd}
\mathfrak{A} \arrow[dd, "\phi"'] \arrow[r] \arrow[r] \arrow[rr, bend left, "h"] & \mathscr{S}_0(\mathfrak{A}) \arrow[r]                                                & \mathscr{A}_0(\mathscr{S}_0(\mathfrak{A})) \arrow[dd, "\mathscr{A}_1(\mathscr{S}_1(\phi))"] \\
{} \arrow[r, "1-1", no head, dashed]                                       & {} \arrow[r, "1-1", no head, dashed]                                                 & {}                                                                                          \\
\mathfrak{A}' \arrow[r] \arrow[rr, bend right, "h"']                             & \mathscr{S}_0(\mathfrak{A}') \arrow[uu, "\mathscr{S}_1(\phi)" description] \arrow[r] & \mathscr{A}_0(\mathscr{S}_0(\mathfrak{A}'))                                                
\end{tikzcd}
\hskip 4em
\begin{tikzcd}
\mathcal{X} \arrow[dd, "\psi"'] \arrow[r] \arrow[rr, bend left, "f"] & \mathscr{A}_0(\mathcal{X}) \arrow[r]                                                & \mathscr{S}_0(\mathscr{A}_0(\mathcal{X})) \arrow[dd, "\mathscr{S}_1(\mathscr{A}_1(\psi))"] \\
{} \arrow[r, "1-1", no head, dashed]                            & {} \arrow[r, "1-1", no head, dashed]                                                & {}                                                                                         \\
\mathcal{X}' \arrow[r] \arrow[rr, bend right, "f"']                   & \mathscr{A}_0(\mathcal{X}') \arrow[uu, "\mathscr{A}_1(\psi)" description] \arrow[r] & \mathscr{S}_0(\mathscr{A}_0(\mathcal{X}'))                                                
\end{tikzcd}

\noindent where $\mathscr{A}_1(\mathscr{S}_1(\phi))=(\phi^{-1})^{-1}$ and $\mathscr{S}_1(\mathscr{A}_1(\psi))=(\psi^{-1})^{-1}$ since $\mathscr{S}_1(\phi)=\phi^{-1}$ and $\mathscr{A}_1(\psi)=\psi^{-1}$. The dotted lines in the above commutative diagrams labelled \say{$1-1$} simply indicate the $1-1$ (dual) correspondence that has been established between the respective morphisms. This completes the proof that the contravariant functors $\mathscr{A}_F\colon\mathbb{OMS}\to\mathbb{OML}$ and $\mathscr{S}_F\colon\mathbb{OML}\to\mathbb{OMS}$ constitute a dual equivalence of categories.  
\end{proof}
\section{Algebraic and topological semantics for quantum logic}

We conclude by studying an application of our duality within the setting of the quantum logic $\mathcal{Q}$ as described in \cite{hartonas}. For an alternative formulation of $\mathcal{Q}$, refer to \cite{goldblatt2}. It is well-known that orthomodular lattices provide an algebraic semantics for various formulations quantum logic (such as $\mathcal{Q}$) in an analogous way in which Boolean algebras provide an algebraic semantics for classical logic. Therefore, as an application of our duality for orthomodular lattices, we develop a topological semantics for quantum logic using orthomodular spaces. For other papers which develop a topological semantics for logical calculi by directly exploiting a duality theorem for the Lindenbaum algebra of the logic, we refer the reader to for instance Bezhanishvili, Grilletti, and Holiday in \cite{bez} as well as Massas in \cite{mas}.

\subsection{The quantum logic $\mathcal{Q}$}
\begin{definition}
Let $\mathcal{L}_{\mathcal{Q}}$ be a language consisting of a countable set of propositional variables $\textit{Var}=\{p_1,p_2,\dots\}$ and a collection of $\mathcal{L}_{\mathcal{Q}}$-formulas generated by the following grammar: 
\[\alpha:=F\mid p_i\mid\neg\alpha\mid\alpha\wedge\alpha\]
where $p_i$ rewrites to any element of $\textit{Var}$. 
\end{definition}
We present the quantum logic $\mathcal{Q}$ as a binary logic in the sense of \cite{goldblatt2}. 
\vspace{2cm}
\begin{definition} 
The quantum logic $\mathcal{Q}$ consists of the language $\mathcal{L}_{\mathcal{Q}}$ (as described in Definition 5.1) and a provability relation $\vdash_{\mathcal{Q}}\subseteq\mathcal{L}_{\mathcal{Q}}\times\mathcal{L}_{\mathcal{Q}}$ satisfying the following axioms: 
\begin{multicols}{2}
    \begin{enumerate}
    \item $\alpha\vdash_{\mathcal{Q}}\alpha$
    \item $F\vdash_{\mathcal{Q}}\beta$
    \item $\alpha\wedge\beta\vdash_{\mathcal{Q}}\alpha$
    \item $\alpha\wedge\beta\vdash_{\mathcal{Q}}\beta$
    \item $\alpha\vdash_{\mathcal{Q}}\neg\neg\alpha$
    \item $\neg\neg\alpha\vdash_{\mathcal{Q}}\alpha$
    
    \end{enumerate}
\end{multicols}

The collection of theorems in $\mathcal{Q}$ includes all instances of axioms 1--6 and is closed under the application of the following rules of inference:  

\begin{multicols}{2}
            \begin{enumerate}
    \item[7.] $\alpha\vdash_{\mathcal{Q}}\beta/\neg\beta\vdash_{\mathcal{Q}}\neg\alpha$
    \item[8.] $\alpha\vdash_{\mathcal{Q}}\beta$,\hspace{.1cm} $\beta\vdash_{\mathcal{Q}}\gamma/\alpha\vdash_{\mathcal{Q}}\gamma$
    \item[9.] $\alpha\vdash_{\mathcal{Q}}\beta$,\hspace{.1cm} $\alpha\vdash_{\mathcal{Q}}\gamma/\alpha\vdash_{\mathcal{Q}}\beta\wedge\gamma$
    \item[10.] $\alpha\vdash_{\mathcal{Q}}\beta$,\hspace{.1cm} $\neg\alpha\wedge\beta\vdash_{\mathcal{Q}} F/\beta\vdash_{\mathcal{Q}}\alpha$
\end{enumerate}
\end{multicols}

\end{definition}

\subsection{Algebraic semantics for $\mathcal{Q}$}
In this subsection, we outline the algebraic semantics of $\mathcal{Q}$. 
\begin{definition}
Let $\mathfrak{A}=\langle A;\cdot,-,0\rangle$ be an orthomodular lattice and let $v_{\mathfrak{A}}\colon\textit{Var}\to\mathfrak{A}$ be an atomic valuation map. We extend $v_{\mathfrak{A}}$ to a map $\widehat{v_{\mathfrak{A}}}\colon\mathcal{L}_{\mathcal{Q}}\to\mathfrak{A}$ according to the following recursive definition: 
\begin{multicols}{2}
\begin{enumerate}
     \item $\widehat{v_{\mathfrak{A}}}(F)=0$
    \item $\widehat{v_{\mathfrak{A}}}(p)=v_{\mathfrak{A}}(p)$
    \item $\widehat{v_{\mathfrak{A}}}(\neg\alpha)=-\widehat{v_{\mathfrak{A}}}(\alpha)$
    \item $\widehat{v_{\mathfrak{A}}}(\alpha\wedge\beta)=\widehat{v_{\mathfrak{A}}}(\alpha)\cdot\widehat{v_{\mathfrak{A}}}(\beta)$ 
\end{enumerate}
\end{multicols}

\end{definition}
Let $\mathbb{OML}_0$ denote the class of all orthomodular lattices. 
\begin{definition}
Given formulas $\alpha,\beta\in\mathcal{L}_{\mathcal{Q}}$, define a semantic consequence relation $\models_{\mathbb{OML}_0}\subseteq\mathcal{L}_{\mathcal{Q}}\times\mathcal{L}_{\mathcal{Q}}$ by: \[\alpha\models_{\mathbb{OML}_0}\beta\hspace{.3cm}\Leftrightarrow\hspace{.3cm} \widehat{v_{\mathfrak{A}}}(\alpha)\leq\widehat{v_{\mathfrak{A}}}(\beta)\] for each orthomodular lattice $\mathfrak{A}\in\mathbb{OML}_0$ and every valuation $\widehat{v_{\mathfrak{A}}}\colon\mathcal{L}_{\mathcal{Q}}\to\mathfrak{A}$. 
\end{definition}

\begin{theorem}[Algebraic soundness and completeness of $\mathcal{Q}$]
The quantum logic $\mathcal{Q}$ is sound and complete with respect to $\mathbb{OML}_0$ under the algebraic semantics given in Definitions 5.3 and 5.4, i.e., $\alpha\vdash_{\mathcal{Q}}\beta\Leftrightarrow\alpha\models_{\mathbb{OML}_0}\beta$.  
\end{theorem}
\begin{proof}
We only sketch the proof as the result is well-known. For more details on the algebraic foundations of quantum logics and related calculi, consult \cite{birkhoff,fei}.

For soundness, it is easy to see that $\alpha\vdash_{\mathcal{Q}}\beta$ implies $\alpha\models_{\mathbb{OML}_0}\beta$ for any $\alpha,\beta\in\mathcal{L}_{\mathcal{Q}}$. For completeness, we construct the Lindenbaum algebra $\mathfrak{A}_{\mathcal{Q}}$ of $\mathcal{Q}$ as follows: The carrier set of $\mathfrak{A}_{\mathcal{Q}}$ is the quotient space $\mathcal{L}_{{\mathcal{Q}}_{/\equiv_{\mathcal{Q}}}}$ where $\equiv_{\mathcal{Q}}\subseteq\mathcal{L}_{\mathcal{Q}}\times\mathcal{L}_{\mathcal{Q}}$ is a congruence relation defined by
\[\alpha\equiv_{\mathcal{Q}}\beta\hspace{.3cm}\Leftrightarrow\hspace{.3cm}\alpha\vdash_{\mathcal{Q}}\beta\hspace{.2cm}\&\hspace{.2cm}\beta\vdash_{\mathcal{Q}}\alpha\]

 Now let $[\alpha]$ denote the equivalence class of $\alpha$ under $\equiv_{\mathcal{Q}}$ and define the following type-lifted operators $0:=[F]$, $-[\alpha]:=[\neg\alpha]$, and $[\alpha]\cdot[\beta]:=[\alpha\wedge\beta]$. Moreover, by defining the lattice order of $\mathfrak{A}_{\mathcal{Q}}$ by:
\[[\alpha]\leq[\beta]\hspace{.3cm}\Leftrightarrow\hspace{.3cm}\alpha\vdash_{\mathcal{Q}}\beta\] 
\vspace{2cm}
\noindent then it is easy to see that the Lindenbaum algebra $\mathfrak{A}_{\mathcal{Q}}$ satisfies:
\begin{multicols}{2}
\begin{enumerate}
    \item $[\alpha]\cdot-[\alpha]=0$
    \item $[\alpha]=--[\alpha]$
    \item $[\alpha]\leq[\beta]\hspace{.1cm}\Rightarrow\hspace{.1cm}-[\beta]\leq-[\alpha]$
    \item $[\alpha]\leq[\beta]\hspace{.1cm}\&\hspace{.1cm}-[\alpha]\cdot[\beta]=0\hspace{.1cm}\Rightarrow\hspace{.1cm}[\alpha]=[\beta]$
\end{enumerate}
 
\end{multicols}

Hence the Lindenbaum algebra $\mathfrak{A}_{\mathcal{Q}}$ of $\mathcal{Q}$ is an orthomodular lattice. By setting $v_{\mathfrak{A}}(p_i):=[p_i]$ for $p_i\in\textit{Var}$, an easy induction on the complexity of $\mathcal{L}_{\mathcal{Q}}$-formulas shows that $\widehat{v_{\mathfrak{A}}}(\alpha)=[\alpha]$ for any $\mathcal{L}_{\mathcal{Q}}$-formula $\alpha$. Therefore if we assume $\alpha\not\vdash_{\mathcal{Q}}\beta$, then $[\alpha]\not\leq[\beta]$. This implies that $\widehat{v_{\mathfrak{A}}}(\alpha)\not\leq\widehat{v_{\mathfrak{A}}}(\beta)$ and hence $\alpha\not\models_{\mathbb{OMS}_0}\beta$. This establishes that $\alpha\models_{\mathbb{OMS}_0}\beta$ implies $\alpha\vdash_{\mathcal{Q}}\beta$, as desired.        
\end{proof}

\subsection{Topological semantics for $\mathcal{Q}$}
In light of our duality and the algebraic semantics for $\mathcal{Q}$ described in the Section 5.2, we now develop a topological semantics for $\mathcal{Q}$ and prove soundness and completeness results. 
\begin{definition}
Let $\mathcal{X}=\langle X;\preceq,\perp,\Omega,\mathcal{T}_X\rangle$ be an orthomodular space and let $v_{\mathcal{X}}\colon\textit{Var}\to\mathcal{CO}(\mathcal{X})^{\dagger}$ be an atomic valuation map. We extend the map $v_{\mathcal{X}}$ to a map $\widehat{v_{\mathcal{X}}}\colon\mathcal{L}_{\mathcal{Q}}\to\mathcal{CO}(\mathcal{X})^{\dagger}$ recursively as follows: 

\begin{multicols}{2}
 \begin{enumerate}
     \item $\widehat{v_{\mathcal{X}}}(F)=\{\omega\}$
        \item $\widehat{v_{\mathcal{X}}}(p)=v_{\mathcal{X}}(p)$
        \item $\widehat{v_{\mathcal{X}}}(\neg\alpha)=\widehat{v_{\mathcal{X}}}(\alpha)^{\perp}$
        \item $\widehat{v_{\mathcal{X}}}(\alpha\wedge \beta)=\widehat{v_{\mathcal{X}}}(\alpha)\cap\widehat{v_{\mathcal{X}}}(\beta)$
    \end{enumerate}
\end{multicols}

\end{definition}
Let $\mathbb{OMS}_0$ denote the class of all orthomodular spaces. The notion of topological semantic consequence is defined analogously to the algebraic semantic consequence relation given in Definition 5.4. 
\begin{definition}
Given $\mathcal{L}_{\mathcal{Q}}$-formulas $\alpha$ and $\beta$, define a semntic consequence relation $\models_{\mathbb{OMS}_0}\subseteq\mathcal{L}_{\mathcal{Q}}\times\mathcal{L}_{\mathcal{Q}}$ by: 
\[\alpha\models_{\mathbb{OMS}_0}\beta\hspace{.3cm}\Leftrightarrow\hspace{.3cm}\widehat{v_{\mathcal{X}}}(\alpha)\subseteq\widehat{v_{\mathcal{X}}}(\beta)\] for every orthomodular space $\mathcal{X}\in\mathbb{OMS}_0$ and each valuation $\widehat{v_{\mathcal{X}}}\colon\mathcal{L}_{\mathcal{Q}}\to\mathcal{CO}(\mathcal{X})^{\dagger}$. 
\end{definition}

\begin{theorem}[Topological soundness and completeness for $\mathcal{Q}$] $\mathcal{Q}$ is sound and complete with respect to $\mathbb{OMS}_0$ under the topological semantics given in Definitions 5.6 and 5.7, i.e., $\alpha\vdash_{\mathcal{Q}}\beta\Leftrightarrow\alpha\models_{\mathbb{OMS}_0}\beta$. 
\end{theorem}
\begin{proof}
By our duality, $\mathscr{A}_0(\mathcal{X})=\langle\mathcal{CO}(\mathcal{X})^{\dagger};\cap,\hspace{.05cm}^{\perp},\{\omega\}\rangle$ is an orthomodular lattice whenever $\mathcal{X}$ is an orthomodular space. Moreover, every orthomodular lattice $\mathfrak{A}$ is isomorphic to one of the form $\mathscr{A}_0(\mathcal{X})$ where $\mathcal{X}$ is the orthomodular space dual to $\mathfrak{A}$. This, together with Theorem 5.5, Definition 5.6, and Definition 5.7 implies that for all $\mathcal{L}_{\mathcal{Q}}$-formulas $\alpha$ and $\beta$, we have $\alpha\vdash_{\mathcal{Q}}\beta$ if and only if $\alpha\models_{\mathbb{OMS}_0}\beta$. 
\end{proof}

\section{Conclusions and future work}
The aim of this paper was to explore the possibility of adding a topology (along the lines of Bimb\'o \cite{bimbo}) to the orthomodular frames introduced by Hartonas in \cite{hartonas}.  We have shown that the order relation that \cite{hartonas} uses to deal with orthomodularity does not conflict with the order relation that \cite{bimbo} uses to extend Goldblatt's topological duality theorem in \cite{goldblatt1} from objects to morphisms.

Aside from connections to the duality in \cite{bimbo} that was mentioned at the outset of this work, our duality is also related to Priestley duality in \cite{priestley} for distributive lattices in the sense that both use compact totally order-disconnected spaces  
(see \cite[Definition 9.3.1]{bimbo}). Unlike in Priestley space where two points $x$ and $y$ such that $x\not\preceq y$ are separated by clopen upsets, in orthomodular space, they are separated by clopen $\perp$-stables. In addition, close connections can be made to the duality given by J\'onsson and Tarski in \cite{tar} as well as by Bimb\'o and Dunn \cite{bimbo1}, whereby $n$-ary operators on a lattice are represented by $n+1$-ary relations on the frame reduct of their dual spaces. This was seen, for instance, in the representation of the orthocomplementation operator by the relation $\perp_A\subseteq\mathfrak{F}(\mathfrak{A})\times\mathfrak{F}(\mathfrak{A})$ on the canonical frame underlying the dual space $\mathscr{S}_0(\mathfrak{A})$ of $\mathfrak{A}$.  

Future lines of research might include investigating the topological duality theory of orthomodular lattices with additional operations, such as the operations of Sasaki product and Sasaki hook. This would shed light on the topological semantics of quantum logic with implication.

\end{document}